\documentclass{amsart}

\usepackage{amssymb,stmaryrd,mathrsfs,dsfont,amsmath}

\theoremstyle{plain}
\newtheorem{theorem}{Theorem}[section]
\newtheorem{lemma}[theorem]{Lemma}
\newtheorem{proposition}[theorem]{Proposition}
\newtheorem{corollary}[theorem]{Corollary}

\theoremstyle{definition}

\newtheorem{definition}[theorem]{Definition}

\theoremstyle{remark}
\newtheorem{remark}[theorem]{Remark}

\DeclareMathOperator{\reg}{\mbox{reg}}
\DeclareMathOperator{\ureg}{\mbox{ureg}}

\input xy
\xyoption{all}

\begin{document}

\title[Unit-regular elements in monoids of transformations]
{On unit-regular elements in various monoids of transformations}

\author[Mosarof Sarkar]{\bfseries Mosarof Sarkar}
\address{Department of Mathematics, Central University of South Bihar, Gaya, Bihar, India}
\email{mosarofsarkar@cusb.ac.in}

\author[Shubh N. Singh]{\bfseries Shubh N. Singh}
\address{Department of Mathematics, Central University of South Bihar, Gaya, Bihar, India}
\email{shubh@cub.ac.in}


\begin{abstract}
Let $X$ be an arbitrary set and let $T(X)$ denote the full transformation monoid on $X$. We prove that an element of $T(X)$ is unit-regular if and only if it is semi-balanced. For infinite $X$, we discuss regularity of the submonoid of $T(X)$ consisting of all injective (resp. surjective) transformations. For a partition $\mathcal{P}$ of $X$, we characterize unit-regular elements in the monoid $T(X, \mathcal{P})$, under composition, defined as
\[T(X, \mathcal{P}) = \{f\in T(X)\mid (\forall X_i \in \mathcal{P}) (\exists X_j \in \mathcal{P})\; X_i f \subseteq X_j\}.\]
We also characterize (unit-)regular elements in various known submonoids of $T(X, \mathcal{P})$.
\end{abstract}



\maketitle

\section{Introduction}
We assume that the reader is familiar with the basic terminology of algebraic semigroup theory. An element $x$ of a monoid $M$ is said to be \emph{unit-regular} in $M$ if there exists a unit $u\in M$ such that $xux = x$; and monoids consisting entirely of unit-regular elements are called \emph{unit-regular}. The concept of unit-regularity, which is a stronger form of regularity, first appeared in the context of rings \cite{g-ehr68}. There are several natural and motivating examples of such monoids. The unit-regularity in a monoid has received wide attention by a number of authors (see, e.g., \cite{t-blyth83, y-chaiya19, schen74, h-alar80, j-fount02, jhick97, d-mca76, r-mcfa84, ytira79}).


\vspace{0.1cm}
Throughout this paper, let $X$ be an arbitrary nonempty set, $\mathcal{P} = \{X_i|\; i\in I\}$ be a partition of $X$, $E$ be the equivalence relation induced by $\mathcal{P}$, and $T(X)$ be the full transformation monoid on $X$. It is well-known that $T(X)$ is regular (cf. \cite[p. 63, Exercise 2.6.15]{howie95}). It is proven in \cite[Proposition 5]{h-alar80} that $T(X)$ is unit-regular if and only if $X$ is finite. Pei \cite[Corollary 2.3]{pei-c05} characterized regular elements in the monoid $T(X, \mathcal{P})$, under composition, defined as
\begin{align*}
T(X, \mathcal{P}) &= \{f\in T(X)\mid(\forall X_i \in \mathcal{P})(\exists X_j \in \mathcal{P})\; X_i f \subseteq X_j\}\\
&= \{f\in T(X)\mid (x, y)\in E \Longrightarrow (xf, yf)\in E\}.
\end{align*}
Pei \cite[Proposition 2.4]{pei-c05} also showed that $T(X, \mathcal{P})$ is regular if and only if $\mathcal{P}$ is trivial. For finite $X$, the authors \cite[Theorem 3.4]{shubh-c21} characterized unit-regular elements in $T(X, \mathcal{P})$. Several interesting results of $T(X, \mathcal{P})$ and its submonoids have also been investigated (see. e.g., \cite{arjo15, arjo-c04, arjo09, ldeng-s10, ldeng-s12, east-c16, east16, v-ferna-c11, v-ferna-c14, hpei-s94, pei-s98, pei-s05, pei-c05, hpei-ac11, hpeidin-s05, shubh-c20, lsun-s07, lsun-jaa13}).

\vspace{0.1cm}

In recent years, the submonoids
\begin{align*}
\Sigma(X,\mathcal{P}) & = \{f\in T(X, \mathcal{P})\mid (\forall X_i \in \mathcal{P})\; Xf \cap X_i \neq \emptyset\},\\\
\Gamma(X, \mathcal{P}) &= \{f\in T(X, \mathcal{P})\mid (\forall X_i\in \mathcal{P})(\exists X_j\in \mathcal{P})\; X_i f = X_j\},\\\
T_{E^*}(X) &= \{f\in T(X)\mid (\forall x, y\in X) (x, y)\in E \Longleftrightarrow (xf, yf)\in E\}
\end{align*}
of $T(X, \mathcal{P})$ have also been considered. Moreover, the regular elements in $\Gamma(X, \mathcal{P})$ and $T_{E^*}(X)$ have been characterized (cf. \cite[Theorem 5.5]{shubh-c21} and \cite[Theorem 3.1]{ldeng-s10}, respectively). For finite $X$, the authors \cite[Proposition 3.5]{shubh-c21} characterized unit-regular elements in $\Sigma(X,\mathcal{P})$ and showed in \cite[Proposition 5.8]{shubh-c21} that every regular element in $\Gamma(X, \mathcal{P})$ is unit-regular.

\vspace{0.1cm}
Let
\begin{align*}
\Omega(X,\mathcal{P}) &= \{f\in T(X, \mathcal{P})\mid (\forall x, y\in X, x\neq y)(x, y)\in E \Longrightarrow xf \neq yf\}.
\end{align*}
It is routine to verify that $\Omega(X,\mathcal{P})$ is a submonoid of $T(X, \mathcal{P})$. Clearly,  if $E = \{(x, x)\mid x\in X\}$, then $ \Omega(X,\mathcal{P}) = T(X)$; and if $E = X \times X$, then $\Omega(X,\mathcal{P}) = \{f\in T(X)\mid f \mbox{ is injective}\}$.



\vspace{0.1cm}
The rest of this paper is structured as follows. In the next section, we provide definitions and notation used throughout the paper. In Section $3$, we prove that an element of $T(X)$ is unit-regular if and only if it is semi-balanced. For infinite $X$, we discuss regularity of the submonoids $\Omega(X) = \{f\in T(X) \mid f \mbox{ is injective}\}$ and $\Gamma(X) = \{f\in T(X)\mid f \mbox{ is surjective}\}$ of $T(X)$. In Section $4$, we characterize unit-regular elements in $T(X, \mathcal{P})$. In Section $5$, we first characterize regular elements in $\Sigma(X,\mathcal{P})$ and then characterize unit-regular elements in
$\Sigma(X,\mathcal{P})$, $\Gamma(X, \mathcal{P})$, and  $T_{E^*}(X)$. In Section $6$, we finally characterize regular and unit-regular elements in $\Omega(X,\mathcal{P})$.

\section{Preliminaries and Notation}


Let $X$ be a set. The cardinality of $X$ is denoted by $|X|$. For $A, B \subseteq X$, we write $A\setminus B$ to denote the set of elements of $A$ that are not in $B$.  A \emph{partition} of $X$ is a collection of nonempty disjoint subsets, called \emph{blocks}, whose union is $X$. A partition is called \emph{trivial} if it has only singleton blocks or a single block. We denote by $\Delta(X)$ the equivalence relation $\{(x,x)\mid x\in X\}$ on $X$. A \emph{cross-section} of an equivalence relation $E$ on $X$ is a subset of $X$ which intersects every equivalence class of $E$ in exactly one point. The letter $I$ will be reserved for an indexing set.

\vspace{0.1cm}
We write maps to the right of their arguments, and compose them accordingly. The composition of maps will be denoted simply by juxtaposition. Let $f\colon X \to Y$ be a map. We denote by $Af$ the image of a subset $A \subseteq X$ under $f$. The preimage of a subset $B\subseteq Y$ under $f$ is denoted by $Bf^{-1}$. We write $\mbox{dom}(f)$, $\mbox{codom}(f)$, and $\mbox{im}(f)$ for the domain, codomain, and image of $f$, respectively. The \emph{defect} of $f$, denoted by $\mbox{d}(f)$, is defined by $\mbox{d}(f) = |Y\setminus \mbox{im}(f)|$. The \emph{kernel} of $f$, denoted by $\ker(f)$, is an equivalence relation on $X$ defined by $\ker(f) = \{(x, y)\in X\times X|\; xf = yf\}$. We denote by $\pi(f)$ the partition of $X$ induced by $\ker(f)$, namely $\pi(f) = \{y f^{-1}|\; y \in Xf\}$. Note that a map $f$ is injective if and only if $\ker(f) = \Delta(X)$. We write $T_{f}$ to denote any cross-section of $\ker(f)$. Note that $|T_f| = |\mbox{im}(f)|$. The \emph{collapse} of $f$, denoted by $\mbox{c}(f)$, is defined by $\mbox{c}(f) = |X\setminus T_f|$ (cf. \cite[p. 1356]{hig-how-rus98}). If $A \subseteq X$ and $B\subseteq Y$ such that $Af \subseteq B$, then there is a map $g \colon A \to B$ defined by $xg = xf$ for all $x\in A$ and, in this case, we say that the map $g$ \emph{is induced by} $f$.

\vspace{0.1cm}
Let $M$ be a monoid with identity $e$. An element $x\in M$ is called \emph{unit} if there exists $y\in M$ such that $xy = yx = e$. The set $U(M)$ of all units in $M$ forms a subgroup of $M$, and is called the \emph{group of units} of $M$. An element $x\in M$ is said to be \emph{regular} in $M$ if there exists $y\in M$ such that $xyx = x$. Such an element $y$ is called an \emph{inner inverse} for $x$, and it need not be unique. Denote $U(x) = \{y\in U(M)\;|\; x = xyx\}$. If in addition, the inner inverse $y$ is a unit, then $x$ is said to be \emph{unit-regular} in $M$. By $\mbox{reg}(M)$ and $\mbox{ureg}(M)$, we mean the sets of all regular and unit-regular elements in $M$, respectively. If $\mbox{reg}(M) = M$ (resp. $\mbox{ureg}(M) = M)$, we say that the monoid $M$ is \emph{regular} (resp. \emph{unit-regular}). Note that $U(M) \subseteq \mbox{ureg}(M) \subseteq \mbox{reg}(M)$, and if $M$ is  unit-regular then any submonoid of $M$ containing $U(M)$ is also unit-regular.

\vspace{0.1cm}
Let $X$ be a nonempty set. A \emph{selfmap} on $X$ is a map from $X$ into $X$. A \emph{permutation} of $X$ is a bijective selfmap on $X$. We say that a selfmap $f$ on $X$ \emph{preserves} a partition $\mathcal{P}$ of $X$ if for every $X_i \in \mathcal{P}$, there exists $X_j \in \mathcal{P}$ such that $X_i f \subseteq X_j$. We denote by $S(X)$ the symmetric group on $X$. We write $\Omega(X)$ (resp. $\Gamma(X)$) to denote the submonoids of $T(X)$ consisting of all injective (resp. surjective) selfmaps on $X$. Note that the monoids $T(X)$, $\Omega(X)$, and $\Gamma(X)$ have the symmetric group $S(X)$ as their group of units; and $\Omega(X) =\Gamma(X) = S(X)$ if and only if $X$ is finite.

\vspace{0.1cm}
We refer the reader to the standard book \cite{howie95} for additional information from algebraic semigroup theory.

\vspace{0.1cm}
\section{Unit-regular elements in $T(X)$}
In this section, we characterize the unit-regular elements in $T(X)$. We then provide an alternative proof of Proposition 5 in \cite{h-alar80} which characterizes the unit-regularity of $T(X)$. We also characterize the regularity of $\Omega(X)$ and $\Gamma(X)$. We begin by proving the following lemma.

\begin{lemma}\label{cross-section-ker}
Let $f\colon X\to Y$ and $g\colon Y\to X$ be maps. If
$fgf=f$, then $\mbox{im}(fg)$ is a cross-section of the equivalence relation $\ker(f)$.
\end{lemma}

\begin{proof}
Let $yf^{-1} \in \pi(f)$. It suffices to show that $|(Xf)g \cap yf^{-1}| = 1$. Since $y\in Xf$, there exists $x\in X$ such that $xf=y$. Write $yg = z$. Clearly $z = yg \in (Xf)g$. Since $fgf = f$, we obtain \[zf= (yg)f = (xf)gf = xf = y.\] It follows that $z\in yf^{-1}$, so $|(Xf)g\cap yf^{-1}| \ge 1$.

\vspace{0.1cm}
Moreover, suppose that $z_1, z_2 \in (Xf)g\cap yf^{-1}$. Then $z_1f =y =z_2f$. As $z_1, z_2 \in (Xf)g$, there exist $y_1,y_2\in Xf$ such that $y_1 g= z_1$ and $y_2 g = z_2$. Since $y_1,y_2\in Xf$, there exist $x_1, x_2\in X$ such that $x_1f = y_1$ and $x_2 f = y_2$. Since $fgf = f$, we obtain \[y_1 = x_1 f = (x_1f)gf = (y_1g)f = z_1 f = y.\] Similarly, we obtain $y_2 = y$. It follows that $y_1 = y_2$, so $z_1 = z_2$. Hence $|(Xf)g\cap yf^{-1}|= 1$. This completes the proof.
\end{proof}

\vspace{0.1cm}
The following corollary is a direct consequence of Lemma \ref{cross-section-ker} by replacing $Y$ with $X$.

\begin{corollary}\label{cros-sec-Tx}
Let $f, g\in T(X)$. If $fgf = f$, then $\mbox{im}(fg)$ is a cross-section of the equivalence relation $\ker(f)$.
\end{corollary}


Before we begin to prove one of the main theorems, we recall an important definition from \cite{hig-how-rus98}.

\begin{definition}\cite[p. 1356]{hig-how-rus98}
An element $f$ of $T(X)$ is said to be \emph{semi-balanced} if $\mbox{c}(f) = \mbox{d}(f)$ .
\end{definition}

The following theorem gives a characterization of unit-regular elements in $T(X)$.

\begin{theorem}\label{char-unit-reg-TX}
Let $f\in T(X)$. Then $f\in \ureg(T(X))$ if and only if $f$ is semi-balanced.
\end{theorem}

\begin{proof}
Suppose first that $f\in \ureg(T(X))$. Then there exists $g \in S(X)$ such that $fgf= f$. Therefore, by Corollary \ref{cros-sec-Tx}, the set $\mbox{im}(fg)$ is a cross-section of the equivalence relation $\ker(f)$. Write $\mbox{im}(fg) = T_f$. Since $g$ is bijective, we obtain \[(X\setminus Xf)g= Xg \setminus (Xf)g = X\setminus T_f.\] It follows that $\mbox{c}(f) =|X\setminus T_f|= |X\setminus Xf| = \mbox{d}(f)$, so $f$ is semi-balanced.

\vspace{0.1cm}
Conversely, suppose that $f$ is semi-balanced. Note that $ |T_f| = |Xf|$ and  $|T_f\cap yf^{-1}| = 1$ for all $yf^{-1} \in \pi(f)$. Therefore, we choose a bijection $g_1\colon Xf\to T_f$ such that $yg_1=x$ whenever $xf = y$ and  $x\in T_f$. Since $|X\setminus T_f| =|X\setminus Xf|$, we arbitrarily choose any bijection $g_2\colon X\setminus Xf\to X\setminus T_f$. Using these two bijections, we now define a map $g\colon X\to X$ by
\begin{equation*}
yg=
\begin{cases}
yg_1 & \text{if $y\in Xf$},\\
yg_2 & \text{if $y\in X\setminus Xf$}.
\end{cases}
\end{equation*}
Clearly $g\in S(X)$. One can also verify in a routine manner that $fgf= f$. Hence $f\in \ureg(T(X))$.
\end{proof}


The following simple lemma will be useful in the sequel.

\begin{lemma}\label{inj-trans-X}
Let $f\colon X \to Y$ be a map. Then
\begin{enumerate}
\item[\rm(i)] $f$ is injective if and only if $\mbox{c}(f) = 0$;
\item[\rm(ii)] $f$ is surjective if and only if $\mbox{d}(f)= 0$.
\end{enumerate}
\end{lemma}


If $X$ is finite, then every element of $T(X)$ is semi-balanced (cf. \cite[p. 1356]{hig-how-rus98}). However, we have the following lemma.

\vspace{0.0cm}
\begin{lemma}\label{finite-setminus-equal}
Every element of $T(X)$ is semi-balanced if and only if $X$ is finite.
\end{lemma}

\begin{proof}
Suppose that every element of $T(X)$ is semi-balanced. Assume, to the contrary, that $X$ is an infinite set. Choose an arbitrary element $x$ of $X$. Note that the sets $X$ and $X\setminus \{x\}$ have the same cardinality. Therefore there exists is a bijective map $g\colon X \to X\setminus \{x\}$, so the map $g\colon X \to X$ is injective but not surjective. Then we get $\mbox{c}(g) = 0$ and $\mbox{d}(g)\ge 1$ by Lemma \ref{inj-trans-X}. This leads to a contradiction. Hence $X$ is finite.
	
\vspace{0.1cm}
The converse is immediate.
\end{proof}

\vspace{0.1cm}
From Lemma \ref{finite-setminus-equal} and Theorem \ref{char-unit-reg-TX}, we have the corollary below that characterizes unit-regularity of $T(X)$.

\begin{corollary} \label{unit-reg-TX}
The monoid $T(X)$ is unit-regular if and only if $X$ is finite.
\end{corollary}

\begin{proof}
Suppose that $T(X)$ is unit-regular. Then, by Theorem \ref{char-unit-reg-TX}, every element of $T(X)$ is semi-balanced. Hence $X$ is finite by Lemma \ref{finite-setminus-equal}.

\vspace{0.1cm}
Conversely, suppose that $X$ is finite and let $f\in T(X)$. Then $f$ is semi-balanced by Lemma \ref{finite-setminus-equal}. It follows from Theorem \ref{char-unit-reg-TX} that $f\in \ureg(T(X))$. Since $f$ is an arbitrary element, the monoid $T(X)$ is unit-regular.
\end{proof}

\vspace{0.1cm}
If $X$ is a finite set and $f\in T(X)$, the next proposition gives the cardinality of the set $U(f)$.
\begin{proposition}
Let $X$ be a finite set and $f\in T(X)$. Then  \[|U(f)| = \mbox{d}(f)!\displaystyle\prod_{x\in Xf}|xf^{-1}|.\]
\end{proposition}

\begin{proof}
Let $f\in T(X)$. Since $X$ is finite, there exists $g\in S(X)$ such that $fgf=f$ by Corollary \ref{unit-reg-TX}. Then $\mbox{im}(fg)$ is a cross-section of $\ker(f)$ by Corollary \ref{cros-sec-Tx}. Write $T_f = \mbox{im}(fg)$. Note that $|T_f| = |Xf|$ and $T_f$ depends on $g$. Observe that any such bijection $g \in S(X)$ can be only of the following form:
\begin{equation*}
yg=
\begin{cases}
yg_1 & \text{if $y\in Xf$},\\
yg_2 & \text{if $y\in X\setminus Xf$},
\end{cases}
\end{equation*}
where $g_1\colon Xf\to T_f$ is a bijection defined by $yg_1=x$ whenever $xf = y$ and $x\in T_f$; and $g_2\colon X\setminus Xf \to X\setminus T_f$ is any bijection. Note that for every cross-section $T_f$ of $\ker(f)$ there is unique such bijection $g_1$ and $\mbox{d}(f)!$ number of such bijections $g_2$.
	
\vspace{0.1cm}
By the multiplication principle, the total number of possible cross-sections $T_f$ of $\ker(f)$ is $ \prod_{y\in Xf}|yf^{-1}|$. Hence \[|U(f)| = \mbox{d}(f)!\;\prod_{y\in Xf}|yf^{-1}|\] by the multiplication principle.
\end{proof}

\vspace{0.1cm}

We now characterize regularity of $\Omega(X)$ in the following proposition.

\begin{proposition}\label{inj-reg}
The monoid $\Omega(X)$ is regular if and only if $\Omega(X) = S(X)$.
\end{proposition}

\begin{proof}
Suppose that $\Omega(X)$ is regular. Since $S(X)$ is the group of units of $\Omega(X)$, we have $S(X) \subseteq \Omega(X)$. For the reverse inclusion, let $f\in \Omega(X)$. Then $\mbox{c}(f) = 0$ by Lemma \ref{inj-trans-X}(i). Since $f\in \reg(\Omega(X))$, there exists $g\in \Omega(X)$ such that $fgf=f$. It follows from Corollary \ref{cros-sec-Tx} that $\mbox{im}(fg)$ is a cross-section of $\ker(f)$. Since $\mbox{c}(f) = 0$ and $\mbox{im}(fg)$ is a cross-section of $\ker(f)$, it follows that $(Xf)g = X$ and therefore $\mbox{d}(fg) = 0$. Since $g$ is injective, we subsequently have $\mbox{d}(f) = 0$ by \cite[Theorem 2.2(iv)]{hig-how-rus98}. Hence $f$ is surjective by Lemma \ref{inj-trans-X}(ii) and consequently $f\in S(X)$.

\vspace{0.1cm}
The converse is immediate.
\end{proof}

\vspace{0.1cm}

The next proposition provides a characterization of regularity of $\Gamma(X)$.

\begin{proposition}\label{surj-reg}
The monoid $\Gamma(X)$ is regular if and only if $\Gamma(X) = S(X)$.
\end{proposition}

\begin{proof}
Suppose that $\Gamma(X)$ is regular. Since $S(X)$ is the group of units of $\Gamma(X)$, we have $S(X) \subseteq \Gamma(X)$. For the reverse inclusion, let $f\in \Gamma(X)$. Since $f\in \reg(\Gamma(X))$, there exists $g\in \Gamma(X)$ such that $fgf= f$. It follows from Corollary \ref{cros-sec-Tx} that $\mbox{im}(fg)$ is a cross-section of $\ker(f)$. Since $f, g \in \Gamma(X)$, we obtain $\mbox{im}(fg) = (Xf)g= Xg = X$. But $\mbox{im}(fg)$ is a cross-section of $\ker(f)$, we therefore have $\mbox{c}(f) = 0$. Hence $f$ is injective by Lemma \ref{inj-trans-X}(i) and consequently $f\in S(X)$.

\vspace{0.1cm}
The converse is immediate.
\end{proof}

The following corollary is obvious from Proposition \ref{inj-reg} and \ref{surj-reg} .
\begin{corollary}
Let $M = \Omega(X),\; \Gamma(X)$. If $M$ is regular, then $M$ is unit-regular.
\end{corollary}

\vspace{0.1cm}

\section{Unit-regular elements in $T(X,\mathcal{P})$}
In this section, we characterize unit-regular elements in $T(X,\mathcal{P})$. For finite $X$, recall that the authors \cite[Theorem 3.4]{shubh-c21} characterized unit-regular elements in $T(X, \mathcal{P})$.  We begin by considering an important definition  corresponding to a map of $T(X, \mathcal{P})$ from \cite{puri16}.

\begin{definition} (\cite[p. 220]{puri16})
Let $\mathcal{P}=\{X_i|\;i\in I\}$ be a partition of $X$ and let $f \in T(X,\mathcal{P})$. The \emph{character} of $f$, denoted by $\chi^{(f)}$, is a map $\chi^{(f)}\colon I \to I$ defined by  \[i\chi^{(f)}=j \mbox{ whenever } X_i f\subseteq X_j.\]
\end{definition}

\vspace{0.05cm}

In case of finite $X$, the map $\chi^{(f)}$ has also been studied with the notation $\overline{f}$ (see, e.g., \cite{arjo15, east-c16, east16}). The authors gave a characterization of elements of $\Sigma(X, \mathcal{P})$ and $T_{E^*}(X)$ in terms of $\chi^{(f)}$ in \cite[Theorem 3.2]{shubh-c20} and \cite[Theorem 3.4]{shubh-c20}, respectively.


\vspace{0.2cm}
Note that the monoid $T(I)$ is regular (cf. \cite[p. 63, Exercise 2.6.15]{howie95}), but need not be unit-regular (cf. \cite[Proposition 5]{h-alar80}). We now prove the following lemma.

\begin{lemma}\label{char-unit-reg}
Let $\mathcal{P} = \{X_i|\; i\in I\}$ be a partition of $X$ and let $f\in T(X,\mathcal{P})$. If $f \in \ureg(T(X,\mathcal{P}))$, then $\chi^{(f)}\in \ureg(T(I))$.
\end{lemma}

\begin{proof}
If $f \in \ureg(T(X,\mathcal{P}))$, then there exists $g\in S(X,\mathcal{P})$ such that $fgf=f$. Note that $\chi^{(g)} \in S(I)$ by \cite[Theorem 5.8]{shubh-c20}. We also have from \cite[Lemma 2.3]{puri16} that
$\chi^{(f)}\chi^{(g)}\chi^{(f)} =\chi^{(fgf)} = \chi^{(f)}$.
Hence $\chi^{(f)}\in \ureg(T(I))$.
\end{proof}


Before proving Theorem \ref{char-unit-reg-TXP}, we recall a crucial lemma from \cite{shubh-c20}.

\begin{lemma}\cite[Lemma 5.2]{shubh-c20}\label{fam-func}
Let $\mathcal{P} = \{X_i|\;i\in I\}$ be a partition of $X$ and let $f \in T(X)$. Then $f\in T(X,\mathcal{P})$ if and only if there exists a unique family
\[B(f, I) = \{f_i |\;  f_i \mbox{ is induced by }f\mbox{and } \mbox{dom}(f_i) = X_i\;  \forall i\in I\}.\]
\end{lemma}

\vspace{0.1cm}
The following theorem characterizes unit-regular elements in $T(X,\mathcal{P})$.
\begin{theorem}\label{char-unit-reg-TXP}
Let $\mathcal{P} = \{X_i|\; i\in I\}$ be a partition of $X$ and let $f\in T(X,\mathcal{P})$. Then $f \in \ureg(T(X,\mathcal{P}))$ if and only if
\begin{enumerate}
\item $\chi^{(f)}\in \ureg(T(I))$;
\item for each $j\in I\chi^{(f)}$ there exists $i\in I$ such that
\begin{enumerate}
\item[\rm(i)] $|X_i|=|X_j|$,
\item[\rm(ii)] $X_i f = X_j\cap Xf$, and
\item[\rm(iii)] $\mbox{c}(f_i) = \mbox{d}(f_i)$ where $f_i\in B(f, I)$.
\end{enumerate}
\end{enumerate}
\end{theorem}

\begin{proof}
Suppose first that $f \in \ureg(T(X,\mathcal{P}))$. Then there exists $g\in S(X,\mathcal{P})$ such that $fgf = f$.
\begin{enumerate}
\item Lemma \ref{char-unit-reg} takes care of this part.

\vspace{0.1cm}
\item Let $j\in I\chi^{(f)}$ and let $j\chi^{(g)}=i$.

\vspace{0.1cm}
\begin{enumerate}
\item[\rm(i)] Since $g\in S(X,\mathcal{P})$ and $j\chi^{(g)}=i$, we get $|X_i| = |X_j|$ by \cite[Lemma 3.6(ii)]{shubh-c20}.

\vspace{0.1cm}
\item[\rm(ii)] We first show that $X_if \subseteq X_j\cap Xf$. Since $X_if\subseteq X f$, it suffices to prove that $X_if\subseteq X_j$. Since $j\in I\chi^{(f)}$,
there exists $k\in I$ such that $k\chi^{(f)} = j$ and so $X_kf\subseteq X_j$ by definition of $\chi^{(f)}$. Since $j\chi^{(g)}=i$, we also have $X_jg = X_i$ by definition of $\chi^{(g)}$ and \cite[Lemma 3.6]{shubh-c20}. Then
$X_k f = (X_kf)gf\subseteq (X_jg)f=X_if$.
This implies that $X_k f \subseteq X_j \cap X_i f$. Since $f$ preserves $\mathcal{P}$, we thus get $X_if\subseteq X_j$.

\vspace{0.1cm}
For the reverse inclusion, let $y\in X_j\cap Xf$. Since $y\in Xf$, there exists $x\in X$ such that $xf=y$. Also $y\in X_j$ and $j\chi^{(g)}=i$, we see that $yg\in X_i$. Then
$y = xf = (xf)gf=(yg)f \in X_if$. Hence $X_j\cap Xf\subseteq X_if$ as required.

\vspace{0.1cm}
\item[\rm(iii)] Since $j\chi^{(g)}=i$, we have from \cite[Lemma 3.6]{shubh-c20} that $X_jg=X_i$ and so  $X_jg_j=X_i$ where $g_j\in B(g,I)$.
Also, by (ii), we have $X_if=X_j\cap Xf$ and so $X_if_i\subseteq X_j$ where $f_i\in B(f,I)$. Since $fgf=f$, we see that $f_ig_jf_i=f_i$.
Then $(X_if_i)g_j \subseteq X_i$ is a cross-section of $\ker(f_i)$ by Lemma \ref{cross-section-ker}. Write $(X_if_i)g_j = T_{f_i}$. Since $g\in S(X,\mathcal{P})$, the map $g_j\in B(g, I)$ is bijective by \cite[Theorem 5.8]{shubh-c20}. Therefore $(X_j\setminus X_if_i)g_j= X_jg_j\setminus (X_if_i)g_j=X_i\setminus T_{f_i}$ and hence $\mbox{c}(f_i) = \mbox{d}(f_i)$.
\end{enumerate}
\end{enumerate}

\vspace{0.1cm}
Conversely, suppose that the given conditions hold. By (2), let $T_{\chi^{(f)}} \subseteq I$ be a cross-section of $\ker(\chi^{(f)})$ such that the preimage $j(\chi^{(f)})^{-1}\cap T_{\chi^{(f)}}$ of every element $j\in I\chi^{(f)}$ satisfies all the three conditions of (2). By (1), let $\chi^{(g)}\in S(I)$ be the following bijection that satisfies $\chi^{(f)}\chi^{(g)}\chi^{(f)} = \chi^{(f)}$.

\begin{equation*}
j\chi^{(g)}=
\begin{cases}
i\in T_{\chi^{(f)}} & \text{if $j\in I\chi^{(f)}$ and $i\chi^{(f)} = j$},\\
i\in I\setminus T_{\chi^{(f)}} & \text{if $j\in I\setminus I\chi^{(f)}$ and $|X_i|=|X_j|$}.
\end{cases}
\end{equation*}

\vspace{0.1cm}
Using the above bijection $\chi^{(g)}$, we now construct a bijection of $S(X, \mathcal{P})$ whose character is equal to $\chi^{(g)}$.

\vspace{0.1cm}
Let $j\in I\chi^{(f)}$. Then there exists $i\in T_{\chi^{(f)}}$ such that $i\chi^{(f)} = j$. By (2), we then have $|X_i|=|X_j|$, $X_j\cap Xf=X_if$, and $\mbox{c}(f_i) = \mbox{d}(f_i)$ where $f_i\in B(f, I)$. Therefore, we can choose a bijection $g_j\colon X_j\to X_i$ such that $yg_j=x$ whenever $x \in T_{f_i}$ and $xf_i = y$.

\vspace{0.0cm}
Let $j\in I\setminus I\chi^{(f)}$. Then we arbitrarily choose a bijection $h_j\colon X_j\to X_{j\chi^{(g)}}$. Using these bijections, we define a map $g\colon X\to X$ by
\begin{equation*}
xg=
\begin{cases}
xg_j & \text{if $x\in X_j$ where $j\in I\chi^{(f)}$},\\
xh_j & \text{if $x\in X_j$ where $j\in I\setminus I\chi^{(f)}$}.
\end{cases}
\end{equation*}

Since $\chi^{(g)}\in S(I)$ and each map of $B(g,I)$ is bijective, it follows from \cite[Theorem 5.8]{shubh-c20} that $g\in S(X,\mathcal{P})$. We can also verify in a routine manner that $fgf = f$. Hence $f \in \ureg(T(X,\mathcal{P}))$.
\end{proof}

\section{Unit-regular elements in $\Sigma(X, \mathcal{P})$, $\Gamma(X, \mathcal{P})$, and $T_{E^*}(X)$}

In this section, we first give a characterization of a regular element $f$ in $\Sigma(X,\mathcal{P})$ in term of its character $\chi^{(f)}$. We then characterize unit-regular elements in $\Sigma(X,\mathcal{P})$, $\Gamma(X, \mathcal{P})$, and $T_{E^*}(X)$, respectively.

\begin{proposition}\label{reg-sigma-char}
Let $\mathcal{P} = \{X_i|\;i\in I\}$ be a partition of $X$ and let $f\in \Sigma(X,\mathcal{P})$. Then $f \in \reg(\Sigma(X,\mathcal{P}))$ if and only if $\chi^{(f)} \in \reg(\Gamma(I))$.
\end{proposition}

\begin{proof}
Suppose first that $f \in \reg(\Sigma(X,\mathcal{P}))$. Then there exists $g\in \Sigma(X,\mathcal{P})$ such that $fgf = f$. Therefore \[\chi^{(f)} = \chi^{(fgf)} = \chi^{(f)} \chi^{(g)} \chi^{(f)}\] by \cite[Lemma 2.3]{puri16}. Since $f, g \in \Sigma(X,\mathcal{P})$, we have $\chi^{(f)}, \chi^{(g)} \in \Gamma(I)$ by Theorem \cite[Theorem 3.4]{shubh-c20} and so $\chi^{(f)} \in \reg(\Gamma(I))$.	
	
\vspace{0.1cm}
Conversely, suppose that $\chi^{(f)}\in \reg(\Gamma(I))$. Then there exists $h\in \Gamma(I)$ such that $\chi^{(f)} h \chi^{(f)} = \chi^{(f)}$. Since $\chi^{(f)}\in \reg(\Gamma(I))$, it follows from Proposition \ref{surj-reg} that $\chi^{(f)}\in S(I)$ and so $h=(\chi^{(f)})^{-1}$.
	
\vspace{0.1cm}
For each $i\in I$, let $x_i$ be a fixed element of $X_i$. Moreover, for each $x\in Xf$, let $x'$ be a fixed element of $x f^{-1}$. Define a map $g\colon X \to X$ by
\begin{equation*}
xg=
\begin{cases}
x' & \text{if $x\in X_i\cap Xf$},\\
x_{i(\chi^{(f)})^{-1}} & \text{if $x\in X_i \setminus Xf$}.
\end{cases}
\end{equation*}
	
Observe that $g\in T(X, \mathcal{P})$ and $\chi^{(g)}=h \in \Gamma(I)$. Therefore $g\in \Sigma(X,\mathcal{P})$ by \cite[Theorem 3.2]{shubh-c20}. Since $g$ maps each element $x$ of $Xf$ to a fixed element $x'$ of $xf^{-1}$, we can verify in a routine manner that $fgf=f$. Hence $f\in\reg(\Sigma(X,\mathcal{P}))$.
\end{proof}

We now have the following corollary from Proposition \ref{surj-reg} and Proposition \ref{reg-sigma-char}.

\begin{corollary}\label{fInSig-charIsBij}
Let $\mathcal{P} = \{X_i|\; i\in I\}$ be a partition of $X$ and let $f\in \Sigma(X,\mathcal{P})$. If $f \in \mbox{ureg}(\Sigma(X,\mathcal{P}))$, then $\chi^{(f)} \in S(I)$.
\end{corollary}

\begin{proof}
If $f \in \ureg(\Sigma(X,\mathcal{P}))$, then  $f \in \reg(\Sigma(X,\mathcal{P}))$. Therefore $\chi^{(f)} \in \reg(\Gamma(I))$ by Proposition \ref{reg-sigma-char} and so $\chi^{(f)} \in S(I)$ by Proposition \ref{surj-reg}.
\end{proof}

Let us now make the following trivial remark which will be used in the sequel.

\begin{remark}\label{grp-unit-submon}
Let $M$ be a monoid and let $N$ be a submonoid of $M$. Then
\begin{enumerate}
\item[\rm(i)] $\mbox{ureg}(N) \subseteq \mbox{ureg}(M)$.
\item[\rm(ii)] if $U(M) \subseteq N$, then $N \cap \mbox{ureg}(M) \subseteq \mbox{ureg}(N)$.
\end{enumerate}
\end{remark}

Recall that, for finite $X$, the authors \cite[Proposition 3.5]{shubh-c21} gave a characterization of unit-regular elements in $\Sigma(X, \mathcal{P})$.
The following proposition characterizes unit-regular elements in $\Sigma(X,\mathcal{P})$ for arbitrary $X$.

\begin{proposition}
Let $\mathcal{P} = \{X_i|\; i\in I\}$ be a partition of $X$ and let $f\in \Sigma(X,\mathcal{P})$. Then $f \in \mbox{ureg}(\Sigma(X,\mathcal{P}))$ if and only if
\begin{enumerate}
\item[\rm(i)] $\chi^{(f)} \in \Omega(I)$,
\item[\rm(ii)] $|X_i|=|X_j|$ whenever $i\chi^{(f)}=j$, and $\mbox{c}(f_i) = \mbox{d}(f_i)$ where $f_i\in B(f, I)$.
\end{enumerate}
\end{proposition}

\begin{proof}
Suppose first that $f \in \mbox{ureg}(\Sigma(X,\mathcal{P}))$.
\begin{enumerate}
\item[\rm(i)] Then $\chi^{(f)} \in S(I)$ by Corollary \ref{fInSig-charIsBij}. Since $S(I) \subseteq \Omega(I)$, it follows that $\chi^{(f)} \in \Omega(I)$.
	
\vspace{0.1cm}
\item[\rm(ii)] Let $i\chi^{(f)}=j$. Since $f \in \mbox{ureg}(\Sigma(X,\mathcal{P}))$ and $\Sigma(X,\mathcal{P}) \subseteq T(X,\mathcal{P})$, we have $f \in \mbox{ureg}(T(X,\mathcal{P}))$ and $\chi^{(f)}\in S(I)$ by Remark \ref{grp-unit-submon}(i) and by Corollary \ref{fInSig-charIsBij}, respectively. Then, by Theorem \ref{char-unit-reg-TXP}, we get $|X_i|=|X_j|$ and $\mbox{c}(f_i) = \mbox{d}(f_i)$ where $f_i\in B(f, I)$.
\end{enumerate}
	
\vspace{0.1cm}
Conversely, suppose that the given conditions hold. Since $S(X, \mathcal{P}) \subseteq \Sigma(X, \mathcal{P}) \subseteq T(X,\mathcal{P})$, by Remark \ref{grp-unit-submon}(ii), it suffices to show that $f \in \mbox{ureg}(T(X,\mathcal{P}))$.

\vspace{0.1cm}
Since $f\in \Sigma(X,\mathcal{P})$, we have $\chi^{(f)}\in \Gamma(I)$ by \cite[Theorem 3.2]{shubh-c20}.  By (i), we know that  $\chi^{(f)} \in \Omega(I)$ and so  $\chi^{(f)} \in \Gamma(I)\cap \Omega(I) = S(I)$. This implies that $\chi^{(f)}\in \ureg(T(I))$.

\vspace{0,1cm}
Let $i\chi^{(f)}=j$. Then $X_i f \subseteq X_j$ by definition of $\chi^{(f)}$. Recall that $\chi^{(f)}\in S(I)$. Therefore $X_i f = X_j \cap Xf$. By (ii), we also have $|X_i| = |X_j|$ and $\mbox{c}(f_i) = \mbox{d}(f_i)$ where $f_i\in B(f, I)$. Hence $f\in \ureg(T(X, \mathcal{P}))$ by Theorem \ref{char-unit-reg-TXP} as required.
\end{proof}

\vspace{0.1cm}

For finite $X$, it is proven in \cite[Proposition 5.8]{shubh-c21} that $\reg(\Gamma(X, \mathcal{P})) = \ureg(\Gamma(X, \mathcal{P}))$. The following proposition provides a characterization of unit-regular elements in $\Gamma(X, \mathcal{P})$ for arbitrary $X$.

\begin{proposition}
Let $\mathcal{P} = \{X_i|\; i\in I\}$ be a partition of $X$ and let $f\in \Gamma(X,\mathcal{P})$. Then $f \in \ureg(\Gamma(X,\mathcal{P}))$ if and only if
\begin{enumerate}
\item[\rm(i)] $\chi^{(f)}\in \ureg(T(I))$,
\item[\rm(ii)] for each $j\in I\chi^{(f)}$ there exists $i\in I$ such that the map $f_i\in B(f,I)$ from  $X_i$ to $X_j$ is injective.
\end{enumerate}
\end{proposition}

\begin{proof}
Suppose first that $f \in \ureg(\Gamma(X,\mathcal{P}))$. Then  $f \in \ureg(T(X,\mathcal{P}))$ by Remark \ref{grp-unit-submon}(i).
\begin{enumerate}
\item[\rm(i)] Then Lemma \ref{char-unit-reg} takes care of this part.
		
\vspace{0.1cm}
\item[\rm(ii)] Clearly $f\in \reg(\Gamma(X,\mathcal{P}))$. Hence, by Theorem \cite[Theorem 5.5]{shubh-c21}, for each $j\in I\chi^{(f)}$ there exists $i\in I$ such that the map $f_i\in B(f, I)$ from $X_i$ to $X_j$ is injective.
\end{enumerate}
	
\vspace{0.1cm}
Conversely, suppose that the given conditions hold. Since $f\in \Gamma(X,\mathcal{P})$, we know from \cite[Remark 5.3]{shubh-c21} that every map of $B(f, I)$ is surjective. It follows from (ii) that for each $j\in I\chi^{(f)}$ there exists $i\in I$ such that the map $f_i \in B(f, I)$ from $X_i$ to $X_j$ is bijective. Therefore, for the preimage $i\in j(\chi^{(f)})^{-1}$, we get $|X_i| = |X_j|$, $X_if = X_j\cap Xf$, and $\mbox{c}(f_i) = \mbox{d}(f_i)$ where $f_i \in B(f, I)$. Thus $f\in\ureg(T(X, \mathcal{P}))$ by Theorem \ref{char-unit-reg-TXP}. Since $S(X, \mathcal{P}) \subseteq \Gamma(X, \mathcal{P})$, it follows from Remark \ref{grp-unit-submon}(ii) that $f\in \ureg(\Gamma(X, \mathcal{P}))$.
\end{proof}

The following proposition characterizes unit-regular elements in $T_{E^*}(X)$.
\begin{proposition}
Let $\mathcal{P} = \{X_i|\; i\in I\}$ be a partition of $X$ and let $f\in T_{E^*}(X)$. Then $f \in \ureg(T_{E^*}(X))$ if and only if
\begin{enumerate}
\item[\rm(i)] $\chi^{(f)}\in \Gamma(I)$,
\item[\rm(ii)] $|X_i|=|X_j|$ whenever $i\chi^{(f)}=j$, and $\mbox{c}(f_i)=\mbox{d}(f_i)$ where $f_i\in B(f,I)$.
\end{enumerate}
\end{proposition}

\begin{proof}
Suppose first that $f \in \ureg(T_{E^*}(X))$. Then $f \in \ureg(T(X,\mathcal{P}))$ by Remark \ref{grp-unit-submon}(i).

\begin{enumerate}
\item[\rm(i)] It follows that $\chi^{(f)} \in \ureg(T(I))$ by Lemma \ref{char-unit-reg}. Since $f\in T_{E^*}(X)$, we have $\chi^{(f)}\in \Omega(I)$ by \cite[Theorem 3.4]{shubh-c20}. Since $S(I) \subseteq \Omega(I)$, it follows from Remark \ref{grp-unit-submon}(ii) that $\chi^{(f)}\in \ureg(\Omega(I))$ and so $\chi^{(f)}\in S(I)$ by Proposition \ref{inj-reg}. Hence $\chi^{(f)}\in \Gamma(I)$.

\vspace{0.1cm}
\item[\rm(ii)] Let $i\chi^{(f)}=j$. From (i) , we see that $\chi^{(f)}\in S(I)$. Recall that $f\in \ureg(T(X,\mathcal{P}))$. Therefore, by Theorem \ref{char-unit-reg-TXP}, we get $|X_i|=|X_j|$ and $\mbox{c}(f_i)=\mbox{d}(f_i)$ where $f_i\in B(f,I)$.
\end{enumerate}

\vspace{0.1cm}	
Conversely, suppose that the given conditions hold. Since $S(X,\mathcal{P}) \subseteq T_{E^*}(X) \subseteq T(X,\mathcal{P})$, it suffices to show that $f \in \ureg(T(X,\mathcal{P}))$.

\vspace{0.1cm}
Since $f\in T_{E^*}(X)$, we have $\chi^{(f)}\in \Omega(I)$ by Theorem \cite[Theorem 3.4]{shubh-c20}. It follows from (i) that $\chi^{(f)}\in \Gamma(I) \cap \Omega(I) = S(I)$ and so $\chi^{(f)}\in \ureg(T(I))$.
	
\vspace{0.1cm}
Again, since $\chi^{(f)}\in S(I)$, we see that $X_j\cap Xf=X_if$ whenever $i\chi^{(f)}=j$. Combining these with condition (ii), Theorem \ref{char-unit-reg-TXP} concludes that $f \in \ureg(T(X,\mathcal{P}))$ as required.
\end{proof}


\section{The Submonoid $\Omega(X, \mathcal{P})$}

In this section, we characterize regular and unit-regular elements in $\Omega(X, \mathcal{}P)$, respectively. Let us first make the following useful remark.

\begin{remark}\label{Omega-inj-block}
Let $f\in T(X, \mathcal{P})$. Then $f\in \Omega(X, \mathcal{P})$ if and only if every map in $B(f, I)$ is injective.
\end{remark}

The next theorem characterizes regular elements in $\Omega(X,\mathcal{P})$.

\begin{theorem} \label{omega-reg}
Let $\mathcal{P}=\{X_i|\; i\in I\}$ be a partition of $X$ and let $f\in \Omega(X,\mathcal{P})$. Then $f\in \reg(\Omega(X,\mathcal{P}))$ if and only if for every $j\in I\chi^{(f)}$ there exists $i\in I$ such that the map $f_i \in B(f,I)$ from $X_i$ to $X_j$ is surjective.
\end{theorem}
\begin{proof}
Suppose first that $f\in \reg(\Omega(X,\mathcal{P}))$, and let $j\in I\chi^{(f)}$. Then there exists $i\in I$ such that $i\chi^{(f)}=j$ and so $X_i f \subseteq X_j$ by definition of $\chi^{(f)}$. Then there is a map $f_i \in B(f,I)$ from $X_i$ to $X_j$. If $f_i$ is surjective, then we are done. Otherwise, suppose that $f_i$ is not surjective.
	
\vspace{0.1cm}
Recall that  $f\in \reg(\Omega(X,\mathcal{P}))$. There exists $g\in \Omega(X,\mathcal{P})$ such that $fgf = f$. Since $j\in I\chi^{(f)}\subseteq I$, there exists $l\in I\chi^{(g)}$ such that $j\chi^{(g)}=l$ and so $X_j g \subseteq X_l$ by definition of $\chi^{(g)}$. Then there is a map $g_j \in B(g,I)$ from $X_j$ to $X_l$. Recall that $X_if\subseteq X_j$, $X_jg\subseteq X_l$, and $fgf = f$. We obtain \[X_if= (X_if)gf\subseteq (X_jg)f\subseteq X_lf.\] Since $f$ preserves $\mathcal{P}$, we get $X_lf\subseteq X_j$. Therefore there is a map $f_l\in B(f,I)$ from $X_l$ to $X_j$. We now claim that $f_l$ is surjective.

\vspace{0.1cm}
Assume, to the contrary, that there exists $y\in X_j\setminus X_lf_l$. Write $yg_j=z$ and $zf_l=w$. Clearly $z\in X_l$,  $w\in X_j$, and $y\neq w$.
Since $f, g\in \Omega(X, \mathcal{P})$, the maps $g_j$ and $f_l$ are injective by Remark \ref{Omega-inj-block} and so $(wg_j)f_l\neq (yg_j)f_l$. But $(yg_j)f_l=zf_l=w$, and
\[(wg_j)f_l = (wg_j)f = (wg)f = (zf_l)gf= (zf)gf= zf =  zf_l = w\] which is a contradiction. Hence the map $f_l\in B(f, I)$ is surjective as required.

\vspace{0.1cm}
Conversely, suppose that the given condition holds. By Remark \ref{Omega-inj-block}, we know that every map in $B(f,I)$ is injective. Then, by hypothesis, for every $j\in I\chi^{(f)}$ there exists $i\in I$ such that the map $f_i\in B(f, I)$ from $X_i$ to $X_j$ is bijective. Denote by $h_j$ the inverse map of $f_i\colon X_i\to X_j$. Note that the map $h_j\colon X_j\to X_i$ is bijective. We now define a map $g\colon X\to X$ by

\begin{equation*}
xg=
\begin{cases}
xh_j & \text{if $x\in X_j$ where $j\in I\chi^{(f)}$},\\
x & \text{otherwise}.
\end{cases}
\end{equation*}
Clearly $g\in \Omega(X,\mathcal{P})$. We can verify in a routine manner that $fgf = f$. Hence $f\in \reg(\Omega(X,\mathcal{P}))$.
\end{proof}


The following proposition characterizes unit-regular elements in $\Omega(X,\mathcal{P})$.

\begin{proposition}
Let $\mathcal{P} = \{X_i|\; i\in I\}$ be a partition of $X$ and let $f\in \Omega(X,\mathcal{P})$. Then $f\in \ureg(\Omega(X,\mathcal{P}))$ if and only if
\begin{enumerate}
\item[\rm(i)] $\chi^{(f)}\in \ureg(T(I))$, and
\item[\rm(ii)] for each $j\in I\chi^{(f)}$ there exists $i\in I$ such that the map $f_i\in B(f,I)$ from $X_i$ to $X_j$ is surjective.
\end{enumerate}
\end{proposition}
\begin{proof}
Suppose first that $f\in \ureg(\Omega(X,\mathcal{P}))$. Then $f\in \ureg(T(X,\mathcal{P}))$ by Remark \ref{grp-unit-submon}(i).

\begin{enumerate}
\item[\rm(i)] Then Lemma \ref{char-unit-reg} takes care of this part.

\vspace{0.1cm}
\item[\rm(ii)] Note that $f \in \reg(\Omega(X, \mathcal{P}))$ and hence Theorem \ref{omega-reg} takes care of this part.
\end{enumerate}
	
\vspace{0.1cm}
Conversely, suppose that the given conditions hold. Since $f\in \Omega (X,\mathcal{P})$, every map in $B(f,I)$ is injective by Remark \ref{Omega-inj-block}. Then, by hypothesis, for each $j\in I\chi^{(f)}$ there exists $i\in I$ such that the map $f_i\in B(f,I)$ from $X_i$ to $X_j$ is bijective and so $|X_i|=|X_j|$, $X_i f = X_j\cap Xf$, and $c(f_i)=d(f_i)$. Hence $f\in \ureg(T(X,\mathcal{P}))$ by Theorem \ref{char-unit-reg-TXP}. Since $S(X, \mathcal{P}) \subseteq \Omega(X, \mathcal{P})  \subseteq T(X, \mathcal{P})$, we thus get $f\in \ureg(\Omega(X,\mathcal{P}))$ by Remark \ref{grp-unit-submon}(ii).
\end{proof}


\end{document}